\documentclass[11pt]{article}
\usepackage{amsmath}
\usepackage{graphicx}
\usepackage{enumerate}
\usepackage[top=1in,bottom=1in,left=1in,right=1in]{geometry}

\usepackage[numbers]{natbib}
\usepackage[colorlinks,citecolor=blue,urlcolor=blue]{hyperref}
\usepackage{authblk}
\usepackage{amsmath,amsthm,mathrsfs,amsfonts}	
\usepackage{url}
\usepackage{amsfonts}
\usepackage{amssymb}
\usepackage{bbm}
\usepackage{graphicx}
\usepackage{subfig}
\usepackage[normalem]{ulem} 
\usepackage{float}
\usepackage{xcolor}
\usepackage{graphicx}
\usepackage{multirow}
\usepackage{tikz}

\usepackage{url}
\usepackage{bbm}
\usepackage{graphicx}
\usepackage{subfig}
\usepackage{xcolor}
\definecolor{green}{RGB}{255,0,0}
\usepackage{tikz}
\usetikzlibrary{shapes,arrows}
\usetikzlibrary{intersections}

\newcommand{\M}{\mathcal{M}}
\newcommand{\E}{\mathbb{E}}

\newcommand{\eps}{\varepsilon}

\theoremstyle{plain}

\newtheorem{theorem}{Theorem}[section]
\newtheorem{lemma}[theorem]{Lemma}
\newtheorem{remark}{Remark}
\newtheorem{proposition}{Proposition}
\newtheorem{corollary}{Corollary}
\theoremstyle{remark}
\newtheorem{definition}[theorem]{Definition}

\usepackage{booktabs}

\usepackage{todonotes} 

\begin{document}

  \title{\bf On standardness: estimation of the standardness constant and decidability aspects}
  %On standardness and the non-estimability of certain functionals of a set}
  
%\date{}

\author{
    Alejandro Cholaquidis \\ 
    \vspace{-0.7em}
    Facultad de Ciencias, Universidad de la Rep\'{u}blica (Uruguay)  
    \and
    Leonardo Moreno \\ 
    \vspace{-0.7em}
    Departamento de M\'{e}todos Cuantitativos, Facultad de Ciencias Econ\'{o}micas y de Administraci\'{o}n, Universidad de la Rep\'{u}blica (Uruguay) \\ 
    \and
    Beatriz Pateiro-L\'{o}pez \\ 
    \vspace{-0.7em}
    CITMAga (Galician Center for Mathematical Research and Technology) \\ 
    Departamento de Estat\'{i}stica, An\'{a}lise Matem\'{a}tica e Optimizaci\'{o}n, \\ 
    Universidade de Santiago de Compostela (Spain)
}

\maketitle

\begin{abstract}
Standardness is a popular assumption in the literature on set estimation. It also appears in statistical approaches to topological data analysis, where it is common to assume that the data were sampled from a probability measure that satisfies the standard assumption. Relevant results in this field, such as rates of convergence and confidence sets, depend on the standardness parameter, which in practice may be unknown. In this paper, we review the notion of standardness and its connection to other geometrical restrictions. We prove the almost sure consistency of a plug-in type estimator for the so-called standardness constant, already studied in the literature. We propose a method to correct the bias of the plug-in estimator and corroborate our theoretical findings through a small simulation study.  We also show that it is not possible to determine, based on a finite sample, whether a probability measure satisfies the standard assumption.
\end{abstract}

\noindent%

\vfill

\section{Introduction}

Set estimation deals with the problem of estimating a set, or a functional of the set, from a sample of random
points. The set in question could represent a wide range of objects, such as the support of a probability distribution, its boundary or levels sets, among others. In this context, shape restrictions refer to limitations or constraints placed on the shapes of the sets being estimated. These restrictions can take many forms, such as requiring that the sets be convex, star-shaped, or have some other geometric property. The purpose of these restrictions is typically to make the estimation problem easier or more tractable, by reducing the space of possible sets that need to be considered. Standardness is one such shape restriction that is commonly used in set estimation. 
 
 In its simplest version, a set $S\subset \mathbb{R}^d$ is standard with respect to the Lebesgue measure $\mu$ if there exist $\delta>0$ and $\lambda>0$, such that, for all $x\in S$ and $0< \epsilon<\lambda$, the closed ball $\mathrm{B}(x,\epsilon)$, with center $x$ and radius $\epsilon$, satisfies
\begin{equation}\label{st1}
	\mu(\mathrm{B}(x,\epsilon)\cap S)\geq \delta \mu(\mathrm{B}(x,\epsilon)).
\end{equation}

Thus, $S$ is standard with respect to $\mu$ if a constant fraction of each small ball centred at $S$ is contained in $S$. This prevents the set for having sharp outward peaks, as shown in Figure \ref{fig:stand}. We will use the term {\emph{standardness constant}} to  refer to the supremum of the values of $\delta$ that satisfy the standardness condition on $S$ for some $\lambda > 0$.  
Standardness can also be defined more generally by substituting the Lebesgue measure on the left-hand side of (\ref{st1}) with another appropriate measure $\nu$.  We will employ this broader notion of standardness (formally defined in Section \ref{sec:back}) throughout this paper. Specifically, we focus on the case where $\nu$ is a probability measure supported on $S$. Hence, standardness can be viewed as both a geometric and probabilistic condition. We investigate the problem of estimating the standardness constant based on a random sample of independent observations drawn from $\nu$. Initially, we examine a plug-in type estimator and demonstrate its almost sure consistency under very general restrictions on $S$ and $\nu$.  
This estimator has previously been studied in \cite{fasy2014confidence} within the context of Topological Data Analysis (TDA), where its convergence in probability was established under stronger assumptions. We conduct a small simulation study that provides compelling evidence that the performance of this estimator is poor, even for large sample sizes. Then, we propose a method to correct the bias of the plug-in estimator and establish the almost sure consistency of this modified estimator. Our simulation study results demonstrate that the proposed correction significantly improves the performance of the estimator.

\begin{figure}[!h]
    \centering
\begin{tikzpicture}[x=1cm,y=1cm]
    \draw[line width=1pt, fill=green!10] (0,0) -- (2,0) -- (2,2) -- (1,1) --(0,2) --(0,0)
      -- cycle;
      \node at (1,2.5) {(a)};
\end{tikzpicture}
\hspace{1cm}
\begin{tikzpicture}
  \draw[line width=1pt,fill=green!10] (0,0) arc (270:180:1.3)  -- (-1.3,1.3)  arc (180:0:0.65)  -- (0,1.3)  arc (180:0:0.65)  -- (1.3,1.3) arc (0:-90:1.3) -- cycle;
  \node at (0,2.5) {(b)};
 \end{tikzpicture}
\hspace{1cm}
\begin{tikzpicture}[x=1cm,y=1cm]
   \def\mypath{(0,0) -- +(0,2) arc (180:360:1cm) -- +(0,-2) -- (0,0)}
   \fill[fill=green!10]                                \mypath;
   \draw[line width=1pt]\mypath;
    \node at (1,2.5) {(c)};
\end{tikzpicture}

    \caption{The sets shown in (a) and (b) are standard with respect to $\mu$. The set in (c) is not standard with respect to $\mu$.}
    \label{fig:stand}
\end{figure}
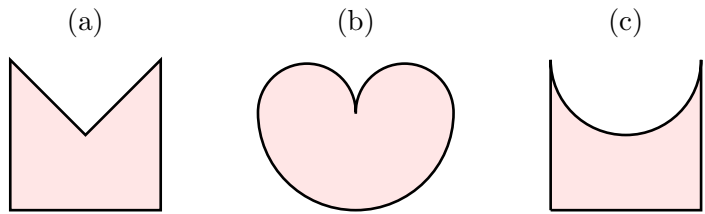

Even though the standardness constant can be consistently estimated, we prove that it is not possible to determine from a sample whether a set is standard or not.  
This finding falls within the realm of a well-known problem in statistics, that addresses the impossibility of consistently estimating certain functionals of a probability measure from a finite sample. A similar problem was recently analysed in \cite{ch2021reach}, where it is shown that it is not possible to determine, based on a finite sample, whether the reach of the support of a density
is zero or not. Determining which functionals can be consistently estimated is a long-standing problem in statistics, with roots dating back to the pioneering works of \cite{bah56} and  \cite{lecam60}. Subsequent studies have attempted to address this problem, including the work of \cite{don88}, as well as more recent contributions such as \cite{fm99} and \cite{romano04}, among others.   In this paper we contribute to this ongoing discussion by approaching the non-estimability problem from a set-estimation perspective.

The paper is organized as follows. In Section \ref{sec:back} we take a closer look at other shape restrictions considered in the literature on set estimation, along with their interdependent associations with standardness. In addition, we provide a comprehensive review of how the concept of standardness has been employed in the set estimation literature, as well as in other domains such as TDA.  In Section \ref{sec-back}, we provide the formal definition of the standardness constant and introduce an additional restriction and some results that will be needed throughout the work. In Section \ref{sec:est_stand} we present the estimators and theoretical results. Some numerical experiments illustrating our results are given in Section \ref{simus}. In Section \ref{noesti} we prove, among others, that it is not possible to determine based on a finite sample whether a set is standard. 
Concluding remarks are provided in Section \ref{sec:conclusion}. All the technical proofs are given in the Appendixes.

\section{Background and related work}\label{sec:back}

In this section, we provide background information and review relevant prior work that contextualizes our research. First, we introduce some notation that will be used throughout the paper. Then, we discuss some of the most common shape restrictions considered in the literature on set estimation, in order of decreasing restrictiveness: convexity, positive reach, $r$-convexity, $r$-rolling condition, and cone-convexity. We then present the formal definition of standardness and review its connection to the aforementioned shape restrictions. Additionally, we discuss the crucial role of standardness in the context of set estimation and TDA.

\subsection{Some notation}
Given a set $S\subset \mathbb{R}^d$, we will denote by
${\textnormal{int}}(S)$, $\overline{S}$ and $\partial S$ the interior, closure and boundary of $S$,
respectively, with respect to the usual topology of $\mathbb{R}^d$. 
 
If $A\subset\mathbb{R}^d$ is a Borel set, then $\mu(A)$ will denote 
its $d$-dimensional Lebesgue measure. We will denote by $\mathrm{B}(x,\varepsilon)$ the closed ball
in $\mathbb{R}^d$, of radius $\varepsilon$, centred at $x$, and by $\omega_d$ the $d$-dimensional Lebesgue measure of the unit ball in $\mathbb{R}^d$.  The parallel set of $S$ of radii $\varepsilon$ will be denoted as $\mathrm{B}(S,\eps)$, that is
$\mathrm{B}(S,\eps) =\{y\in{\mathbb R}^d:\ \inf_{x\in S}\Vert y-x\Vert\leq \eps \}$.
Given a measure $\nu$, we use the notation $\nu\ll \mu$ to indicate that  $\nu$ is absolutely continuous with respect to the Lebesgue measure. The support of a probability measure $\nu$ on $\mathbb{R}^d$ is denoted by ${\textnormal{supp}}(\nu)$.
The Hausdorff distance between two compact non-empty sets in $A, C\subset{\mathbb R}^d$ is given by $d_H(A,C) = \inf\{\epsilon>0: A \subset \mathrm{B}(C,\eps),\ C \subset \mathrm{B}(A,\eps)\}$.

\subsection{Shape constraints beyond convexity}
Convexity is a key concept that has been studied extensively in mathematics. In the context of set estimation, convexity is perhaps the most classical geometric restriction. Since the early works in the 1960s, many contributions have been made in this field, with a focus on various aspects including support estimation \cite{kor93, dumbgen96}, volume and boundary measure estimation \cite{bra98, bal16}, or level set estimation \cite{har87, pol95}. We refer to \cite{bru17} for a recent review of methods of set estimation in the convex case. Nonetheless, convexity can be limiting when dealing with problems that involve more flexible shape constraints. To address this limitation, a lot of research has focused on extending the concept of convexity to more general shape constraints. 

\ 

\noindent{\emph{Positive reach.}} The concept of sets of positive reach was first introduced by \cite{federer59}, as a more flexible alternative to the classical geometric restriction of convexity. While in  a convex set all points in the complement admit a unique projection onto it, in a set with positive reach, only the points of the $r$-parallel neighbourhood of the set are required to have a unique projection. For instance, only the set (c) in Figure \ref{fig:stand} has positive reach.
 
%In addition to the definition of reach, \cite{federer59}  obtains a generalization of Steiner's Theorem and demonstrates the existence of curvature measures, as well as other significant results. 
Numerous authors have further explored the properties of sets of positive reach, see \cite{tha08} for a review. An application of the positive reach condition in set estimation can be found in \cite{cue07}, where the goal is the estimation of the boundary measure of a set.  In \cite{ch2021reach}, a universally consistent estimator for the reach of a set is proposed. Recently, a novel estimator for the reach of manifold-valued data was introduced in \cite{berenfeld2022b}. This method assumes that the manifold is of at least class $C^3$ and requires prior estimation using the manifold estimator proposed in \cite{aamari2019}.

\

\noindent{\emph{$r$-convexity.}} The notion of convexity can also be generalized by that of $r$-convexity. A set $S$ is $r$-convex, with $r>0$, if any point of its complement can be separated from it by an open ball of radius $r$. For instance, out of the sets depicted in Figure \ref{fig:stand}, only set (c) satisfies the assumption of $r$-convexity. %One way to understand how the notion of $r$-convexity generalizes that of convexity is through the characterization of convex sets in terms of supporting hyperplanes. In the case of $r$-convex sets, the supporting hyperplanes would be replaced by supporting balls of radius $r$. 
We refer to \cite{wal99} as the first application of this condition on set estimation, specifically in the context of level set estimation. It can be also found in \cite{rc07,pat13, rod16}, in the context of the problem of estimating a set or its boundary from a random sample of points.
 
Proposition 1 in \cite{cue12} proves that every set with reach $r>0$ is also $r$-convex.  Conversely, Theorem 6 in \cite{cue12} proves that if $S$ is a compact $r$-convex set in $\mathbb{R}^2$ fulfilling an additional mild regularity condition, then $S$ has positive reach (not necessarily $r$).

\ 

\noindent{\emph{Rolling condition.}}
Rolling-type conditions and their relation with the smoothness of the boundary of a set were studied in \cite{wal99}.

\begin{definition}\label{def-outroll}  A closed set $S\subset \mathbb{R}^d$ is said to fulfil the (outside) $r$-rolling condition with $r>0$ if for all $s \in \partial S$, there exists some $x\in S^c$,  such that $\mathrm{B}(x,r)\cap\partial S=\{s\}$. 
 \end{definition}
 
An (inside) $r$-rolling condition could be similarly defined by imposing the condition on $\overline{S^c}$. Such double rolling-type smoothness assumption has been considered, for instance, in \cite{rod16} for the estimation of the support, in \cite{ari17, ari19, aar22}, for the estimation of functionals of a set, such as volume or surface area, and in \cite{rod22}, for density level sets estimation. There is a close connection between the $r$-rolling condition and $r$-convexity. Proposition 2 in \cite{cue12} proves that the $r$-convexity implies the  (outside) $r$-rolling condition. In general, the converse implication does not hold. 
 
Also, if both $S$ and $\overline{S^c}$  satisfy the (outside) $r$-rolling condition then, $S$ and $\overline{S^c}$ are both $r$-convex and $\partial S$ has positive reach, see \cite{wal99} and \cite{pat09}.  Of the sets displayed in Figure \ref{fig:stand}, the set in (b) satisfies an (interior) $r$-rolling condition, while the set in (c) satisfies an (outside) $r$-rolling condition.

\ 

\noindent{\emph{Cone-convexity.}} A  set $S$ is said to fulfil the $\rho$-cone-convex property, with $\rho\in(0,\pi]$, if any point $x\in\partial S$ is the vertex of an open finite cone with opening angle $\rho$, which does not intersect $S$. As with $r$-convexity, one could understand cone-convexity as a generalization of convexity, by replacing in this case the supporting hyperplanes by supporting cones. Cone-convex sets may have indentations, with their level of sharpness being constrained by the angle $\rho$. For example, only set (b) in Figure \ref{fig:stand} fails to satisfy the cone-convexity condition. Note that the cone-convexity is a generalization of the rolling-type conditions commented above. If $S$ satisfies the (outside) $r$-rolling condition, then it is $\rho$-cone-convex for $\rho<\pi/2$. The converse implication does not hold. We refer to \cite{chola14} for more details on cone-convexity and its application in set estimation.

\ 

\noindent{\emph{Standardness.}} This restriction, as given in (\ref{st1}) with respect to the Lebesgue measure, leads to a broader family of sets than that of convex sets. Regarding the other above mentioned shape constraints, it can be easily seen that standardness with respect to $\mu$ is implied by, but not equivalent to, an (inside) $r$-rolling condition. There is a more general version of standardness.

\begin{definition}\label{def-stand}  A Borel set $S\subset \mathbb{R}^d$ is said to be standard with respect to a
	Borel measure $\nu$ 	if there exists $\lambda>0$ and $\delta>0$ such that, for all $x\in S$,
	\begin{equation} \label{estandar}
		\nu(\mathrm{B}(x,\eps)\cap S)\geq \delta \mu(\mathrm{B}(x,\eps)),\quad 0<\eps\leq \lambda.
	\end{equation}
 \end{definition}
 
This more general definition of standardness makes more sense in set estimation, where it is typically assumed that $S$ is the support of a probability distribution $\nu$. In that case, we say that $S$ is standard with respect to $\nu$, or alternatively, that $\nu$ is standard. As we said before, when $\nu$ is the uniform distribution, standardness reduces to a shape constraint on the set $S$, since it is satisfied whenever the set is standard with respect to the Lebesgue measure. When the distribution is not uniform, standardness represents an interplay between the geometry of the set and the measure $\nu$, more than just a shape restriction on $S$. Therefore, standardness becomes an additional regularity constraint to the possible shape constraints of the set to be imposed. This condition was first introduced in the context of set estimation in \cite{cuevas1990pattern}. It also appears in other contexts under different
 names. For example, in the study of function spaces (see the definitions of  $d$-measure and $d$-set in \cite{jonsson1984}). In \cite{benjamini2004boundary} this $d$-set restriction is imposed to study the `Hausdorff dimensions of various random sets associated to the Reflected Brownian Motion'. Revisiting its application in set estimation, the standardness condition can be found in \cite{cuevas1997plug,cuevas2004,chola14,aar16, aaron2017}, among many others. Proposition 1 in \cite{aaron2017} proves that, the (inside) $r$-rolling condition implies standardness with respect to measures $\nu$ whose density with respect to the Lebesgue measure is bounded from below by a positive constant. A double standardness assumption (often referred to in the literature as Ahlfors regularity) was considered in \cite{cue13}, for the problem of the estimation of the Minkowski content. A generalization of the standardness assumption to lower-dimensional sets can be found in \cite{rin10}.

\subsection{Standardness in set estimation}
In what follows, $\aleph_n=\{X_1,\dots,X_n\}$ will denote a sample of independent observations generated from a distribution $\nu$ supported on a set $S\subset\mathbb{R}^d$.  This is the most common sampling scheme in set estimation, and the one we will consider in this work.  There are alternatives to this i.i.d. model, such as the Poisson model or the so-called inner-outer model, see \cite{cue09}. %In the Poisson model, the sample is given by observations of a Poisson process whose intensity function is supported on $S$. In the inner-outer model, the data consists  of i.i.d. points, taken in a set containing $S$, and for each observation, we are able to identify whether it belongs to $S$ or not. A data-dependent sampling scheme, also considered in the literature, assumes that the data comes from the trajectory of a stochastic process living on $S$, observed in an interval $[0,T]$. This approach is of paramount importance in ecology, where $S$ can represent the home-range of an animal, and the goal is to estimate $S$ from a set of locations collected over a period of time, see \cite{cho16, ch21, baillo2021} and references therein.

Standardness arises, within the context of set estimation, in the study of the convergence rate of the Hausdorff distance between $S$ and $\aleph_n$. To be more specific, the hypothesis that $S$ is standard with respect to  $\nu$ is used in \cite{cuevas2004} (Theorem 3) to demonstrate that, for any $\epsilon>0$,  
\[\mathbb{P}(d_H(\aleph_n,S)> 2\epsilon)\leq C\epsilon^{-d}\exp(-n\delta\omega_d\epsilon^d),\]
where $C$ is a constant depending on $S$. This result, in turn, allows us to obtain convergence rates for $d_H(\aleph_n,S)$. As expected, the smaller the value of $\delta$ (which can be interpreted in the uniform case as the set $S$ having sharper peaks), the slower the rate.  The standardness condition is also relevant for obtaining convergence rates for the well-known Devroye-Wise estimator, introduced in \cite{luc80}, as well as for other usual set estimators. Given a sequence of positive real numbers,  $\eps_n\to 0$, the Devroye-Wise estimator is given by
\begin{equation*}\label{devwise}
\hat{S}_n=\bigcup_{i=1}^n \mathrm{B}(X_i,\eps_n).
\end{equation*}
 The  parameter $\eps_n$ must be chosen by the practitioner and depends on the standardness of $S$ with respect to $\nu$ (according to \cite{cuevas2004}, it should be larger than $[2\log(n)/(n\delta\omega_d)]^{1/d}$, where $\delta$ is any value such that \eqref{estandar} holds). Then, a proper choice for the constant $\delta$ is crucial, since an underestimation produces a slower convergence. 
 Moreover, several functionals of $S$ can be estimated by means of $\hat{S}_n$ and, therefore, the convergence rates will also depend on the choice of $\delta$. For instance, $\partial S$ can be consistently estimated by $\partial \hat{S}_n$, choosing the value of $\eps_n$ as before. The estimator $\hat{S}_n$ is also used as an input to define an estimator of the $(d-1)$-Lebesgue measure of $\partial S$, see \cite{aar22}. Standardness plays an important role also to determine, from a sample, if the support of the distribution is lower dimensional or not (see \cite{aaron2017} point (4) in Theorem 1).

%\subsection{Topological data analysis}
  
 Standardness has also gained attention recently within the context of TDA,  in persistent homology, where a generalization, called $(a,b)$-standard assumption, is defined, see \cite{cha15, chazal2015subsampling,chazal2016rates}. It is said that a probability measure $\nu$ in $\mathbb{R}^d$, whose support is a compact set $S$, satisfies the $(a,b)$-standard assumption if there exist $a, b>0$, such that, for every $x\in S$ and every $\eps>0$, $\nu(\mathrm{B}(x,\eps))\geq 1 \wedge a\eps^b$ (in the so-called $(a,b,r_0)$-standard assumption, this inequality holds for all $0\leq \eps<r_0$). Note that the $(a,b)$-standard assumption generalizes condition (\ref{estandar}), since it allows for $b>0$ to not necessarily be equal to $d$.
%Let $\mathcal{P}$ denote a certain well-defined persistence diagram associated with the support $S$ of a probability distribution $\nu$. In \cite{fasy2014confidence}, the authors obtain confidence sets for $\mathcal{P}$ based on an estimate $\hat{\mathcal{P}}$, constructed from a sample $\aleph_n$ generated from $\nu$. In the related paper \cite{cha15}, convergence rates of $\hat{\mathcal{P}}$ to $\mathcal{P}$ are obtained.
The connection between TDA and set estimation arises from the fundamental property of stability in persistence diagrams \cite{cha12}. This property allows us to demonstrate that the bottleneck distance (a usual metric on the space of persistence diagrams) between a certain well-defined persistence diagram ${\mathcal{P}}$ associated with the support $S$ of a probability distribution $\nu$ and an estimate $\hat{\mathcal{P}}$ constructed from a sample $\aleph_n$ generated from $\nu$, is upper bounded by $d_H(\aleph_n, S)$, see \cite{fasy2014confidence},  \cite{cha15}. Therefore, the problem of estimating persistence diagrams can be related to the problem of estimating the support of a measure, with the standardness condition playing a crucial role, as previously mentioned.

\section{The standardness constant}\label{sec-back}

In this section, we analyze the formal definition of standardness and introduce an additional constraint necessary for our study. We also present results proving that, under certain conditions, both standardness and this additional constraint are satisfied.

If a set $S\subset \mathbb{R}^d$ is standard with respect to a Borel measure $\nu$, as given in Definition \ref{def-stand}, and \eqref{estandar} is satisfied for some pair $(\delta, \lambda)$, then \eqref{estandar} also holds for any $\delta' < \delta$ with the same $\lambda$. Thus, we will define the standardness constant associated to $S$ as the largest value for $\delta$, for which \eqref{estandar} holds. 

 \begin{definition}\label{def:cte}
     Assume that $S\subset \mathbb{R}^d$ is standard with respect to $\nu$ and let $A$ be the set of all pairs $(\delta, \lambda)\in \mathbb{R}^+\times \mathbb{R}^+$ fulfilling \eqref{estandar}.   We define the standardness constant
\begin{equation}\label{stconst}
	\Upsilon(S,\nu) = \sup_{(\delta, \lambda) \in A} \delta.
	\end{equation}
 \end{definition}

The following shape restriction (slightly more restrictive than standardness, see Remark \ref{rem1})  will be required:

\

{{(H1)}} The support $S\subset\mathbb{R}^d$ of a probability measure $\nu$ satisfies
 
\begin{equation}\label{H1}
	0<\lim_{r\to 0} \inf_{x\in S} \frac{\nu(\mathrm{B}(x,r))}{r^d\omega_d}=:\delta_{\text{est}}<\infty.
\end{equation}

While the standardness constant $\Upsilon(S,\nu)$, as given in Definition \ref{def:cte}, arises naturally from the concept of standardness, the restriction in (H1) and the definition of $\delta_{\text{est}}$ provide a very natural framework for defining a plug-in estimator. Moreover, the results presented later demonstrate that $\Upsilon(S,\nu)$ and $\delta_{\text{est}}$ coincide in many situations of interest. In particular, in case that the limit in \eqref{H1} exists and is finite, the following Lemma states that $S$ is standard with respect to $\nu$, with standardness constant $\delta_{\text{est}}$.

\begin{lemma}\label{lemaleo} Assume that the support $S\subset\mathbb{R}^d$ of a probability measure $\nu$ fulfils condition {\textnormal{(H1)}}. Then, $S$ is standard with respect to  $\nu$ and 
	\begin{equation*}
		\Upsilon(S,\nu)=\lim_{r\to 0} \inf_{x\in S} \frac{\nu(\mathrm{B}(x,r))}{r^d\omega_d},
	\end{equation*}
where $\Upsilon(S,\nu)$ is the standardness constant defined by \eqref{stconst}.
\end{lemma}	

\begin{remark}\label{rem1} The converse implication in Lemma \ref{lemaleo} is not true in general, that is, there exists a set $S$ and a measure $\nu$ such that $S$ is standard with respect to $\nu$ but the limit \eqref{H1} does not exist: consider $S\subset [0,1]$ the ternary Cantor's set and $\nu$ the measure defined by the Cantor's function, restricted to $S$. Standardness follows from Proposition 5.5 in \cite{dov2006}, while the non-existence of the limit in \eqref{H1}  follows from Proposition 9.1 in \cite{dov2006}.
\end{remark}

Next, Proposition \ref{lemgeo} states that if $\nu$ has a continuous density  and $S$ is standard with respect to $\mu$, then $S$ is also standard with respect to $\nu$, and $\Upsilon(S,\nu)$ is either the minimum of the density, 
or is determined by an interplay between the density on the boundary and the standardness of $S$ with respect to $\mu$. This result further supports the previously mentioned observation regarding standardness, which is influenced by both the geometry of the support and the measure itself. First, Lemma \ref{lemauxst} states that if $S$ is compact, then the infimum in \eqref{H1} is attained. The proof is left as an easy exercise.

\begin{lemma}\label{lemauxst} Let $S\subset \mathbb{R}^d$ and $\nu$ a probability distribution supported on $S$ such that $\nu(\partial \mathrm{B}(x,r))=0$ for all $r>0$ and $x\in S$. Then,
	$\nu(\mathrm{B}(x,r))$ is a continuous function of $(x,r)$.
\end{lemma}

\begin{proposition}\label{lemgeo}
	Let $S\subset\mathbb{R}^d$ be a compact set such that
	\begin{equation}\label{prophip}
		0<\mathcal{L}_{\partial S}=\lim_{r\to 0} \min_{x\in \partial S}\frac{\mu(\mathrm{B}(x,r)\cap S)}{r^d\omega_d}<\infty.
	\end{equation}
	Assume that  $\nu\ll \mu$ and the density $f$ of $\nu$ with respect to $\mu$ is continuous and bounded from below on $S$. Then,   $S$ is standard with respect to $\nu$, and condition {\textnormal{(H1)}} is fulfilled. Moreover, 
	\begin{equation}\label{estdens}
		\Upsilon(S,\nu)=\min\Bigg\{ \lim_{r\to 0} \min_{x\in \partial S}  f(x) \frac{\mu(\mathrm{B}(x,r)\cap S)}{r^d\omega_d}, \min_{x\in S} f(x)\Bigg\}.
	\end{equation}
\end{proposition}

The following corollaries are a consequence of Proposition \ref{lemgeo}. Corollary \ref{coro1},  states that, roughly speaking, when $\nu$ is uniformly distributed on $S$, the value of $\Upsilon(S,\nu)$ is determined by the sharpness of the outward peaks.  Corollary \ref{coro2}, states that if the boundary is smooth, $\Upsilon(S,\nu)$ equals $1/(2\mu(S))$.

\begin{corollary} \label{coro1} 	Let $S\subset \mathbb{R}^d$ be a compact set  such that the limit \eqref{prophip} exists. Then $S$ is standard with respect to the measure $\nu=\mu/\mu(S)$ and 
$$\Upsilon(S,\nu)=\lim_{r\to 0}\min_{x\in \partial S} \frac{\mu(\mathrm{B}(x,r)\cap S)}{r^d\omega_d\mu(S)}.$$
\end{corollary}

\begin{corollary} \label{coro2} Let $S\subset \mathbb{R}^d$ such that $S$ fulfils the inside $r$-rolling ball condition for some $r>0$ and define $\nu=\mu/\mu(S)$. Then, 
$$\Upsilon(S,\nu)\geq \frac{1}{2\mu(S)}.$$ 
Moreover, if $\partial S$ fulfils also the outside $r'$-rolling ball condition, for some $r'>0$, $\Upsilon(S,\nu)=1/(2\mu(S)).$ 
\end{corollary}

\section{Estimation of the standardness constant}\label{sec:est_stand}
 
This section presents the proposed estimators and provides the theoretical results regarding their performance. Let $\aleph_n=\{X_1,\dots,X_n\}$ be an i.i.d. sample of a random vector  $X$ whose distribution, $\nu$, is supported on a compact set $S\subset \mathbb{R}^d$. First we prove that, under very general conditions, the plug-in estimator 	
\begin{equation}\label{hatest0}
	\hat{\Upsilon}_n\equiv \hat{\Upsilon}_n(S,\nu)=  \min_j\frac{\#\{\aleph_n\cap \mathrm{B}(X_j,r_n)\}/n}{r_n^d\omega_d}.
\end{equation}
is a.s. consistent. Recall that the consistency in probability of \eqref{hatest0} is obtained in Theorem 5 in \cite{fasy2014confidence}, under the hypotheses that $\nu$ has a density with respect to Lebesgue measure, along with additional assumptions that $\rho(x,t)=\nu(\mathrm{B}(x,t/2))/t^d$ is bounded and continuous as a function of $t$, differentiable for $t\in(0,t_0)$ and right differentiable at $0$. The partial derivative $\partial \rho(x,t)/\partial t$ is assumed to exist and to be bounded away from zero and infinity for $t$ in an open neighbourhood of zero. They also assume that 
$$\sup_x \sup_{0\leq t \leq t_0 }\left| \frac{\partial \rho(x,t)}{\partial t}\right|\leq C_1<\infty 	\quad \text{and}\quad \sup_{0\leq t\leq t_0}\Big |[\inf_x \rho(x,t)]'\Big|\leq C_2<\infty.$$
We prove in Theorem \ref{th:esti} the almost sure consistency of \eqref{hatest0} under weaker assumptions on $S$.

\begin{theorem} \label{th:esti} Let $\aleph_n=\{X_1,\dots,X_n\}$ be an i.i.d. sample of a random vector $X$ whose distribution, $\nu$, is supported on a compact set $S\subset \mathbb{R}^d$, such that $\nu(\partial \mathrm{B}(x,r))=0$ for all $r<r_0$ and $x\in S$. Assume that $S$  fulfils condition {\textnormal{(H1)}}. Denote by 
	$$\Omega(r)=\Bigg|\min_{x\in S} \frac{\nu(\mathrm{B}(x,r)\cap S)}{\mu(\mathrm{B}(x,r))}-\Upsilon(S,\nu)\Bigg|.$$
	Let  $r_n$ be a sequence of positive real numbers such that $r_n\to 0$,  
and $\hat{\Upsilon}_n$ be the estimator defined in \eqref{hatest0}. Then,
\begin{equation}\label{limite}
	\beta_n|\hat{\Upsilon}_n -  \Upsilon(S,\nu)|\to 0 \quad \text{a.s.},
\end{equation}
where $\beta_n\to \infty$ is such that $\beta_n\Omega(r_n^{1+\gamma})\to 0$ for some $\gamma>0$, and 
$$\frac{n^{1/d}r_n^{2d}}{\beta_n^2(\log(n))^{1/d}}\to \infty.$$
	
\end{theorem}

 \begin{remark}
 There are cases in which $\Omega(r_n)=0$ for sufficiently large $n$, for example, when $\partial S$ is polygonal and $\nu$ is the uniform distribution on $S$.  In such cases, from  $(n^{1/d}r_n^{2d})/(\beta_n^2\log(n)^{1/d})\to \infty$, the best possible rate in \eqref{limite} is $\beta_n=(n/\log(n))^{1/(2d)-\tau}$ for any $\tau>0$. However, if a ball rolls freely inside and outside $S$ and the density $f$ of $\nu$ is  H\"older  continuous in $S$ (i.e there exists $L>0$ and $\alpha\in (0,1]$ such that $|f(x)-f(y)|\leq L|x-y|^\alpha$, for all $x,y\in S$)   we have the following result.
 \end{remark}

 \begin{proposition}  \label{omega} Let $S\subset \mathbb{R}^d$  be a compact set fulfilling Assumption {\textnormal{(H1)}} and $\nu$ a probability measure supported on $S$ whose  density, $f$, is H\"older  continuous, with constants $L$ and $\alpha$
 Assume that there exists $x_r\in \arg\min_x \nu(\mathrm{B}(x,r)\cap S))/(\mu(\mathrm{B}(x,r)))$ such that $x_r\to x^*$ and $x^*$ is a local minimum of $f$.  Assume also that a ball of radius $R>0$ rolls freely inside and outside $S$. Then,
 		\begin{equation}\label{propomega}
 		\Omega(r)=\Bigg|\min_{x\in S} \frac{\nu(\mathrm{B}(x,r)\cap S)}{\mu(\mathrm{B}(x,r))}-\Upsilon(S,\nu)\Bigg|={O}(r^\alpha).
 		\end{equation}
 \end{proposition}

 \begin{remark} Regarding the assumptions in Proposition \ref{omega}, the point  $x_r$   exist because,  by Lemma \ref{lemauxst}, $\nu(\mathrm{B}(x,r))/(r^d\omega_d)$ is a continuous  function. On the other hand, under the hypotheses of Proposition \ref{omega}, using that  $\beta_n\to \infty$ is such that $\beta_nr_n^{\alpha(1+\gamma)}\to 0$ for some $\gamma>0$, and $(r_n^{2d}/\beta_n^2)(n/\log(n))^{1/d}\to \infty$ we get that the best possible rate for $\beta_n$ is   
 	$$\beta_n=\left(\frac{n}{\log(n)}\right)^{\frac{\alpha(1+\gamma)}{2d(d+\alpha(1+\gamma))}-\tau}$$ for any $\tau>0$.
 	\end{remark}

\begin{remark}
As we will see in the simulations (Section \ref{simus}), in some cases the performance of the plug-in estimator $\hat{\Upsilon}_n$ is quite poor.  
This happens typically when $\min_{x\in S} \nu(\mathrm{B}(x,r))/\mu(\mathrm{B}(x,r))$ is attained at `too many $x$'. This is the case, for instance, when $S=\mathrm{B}(0,1/\sqrt{\pi})\subset\mathbb{R}^2$ and $\nu$ is the uniform distribution on $S$.  
For $n= 3000, 5000, 9000$ the mean value of $\hat{\Upsilon}_n$ over 500 replications was, respectively, $0.4048, 0.4131, 0.4231$, showing that, even for very large values of $n$, $\hat{\Upsilon}_n$ %underestimates
does not provide an accurate estimate of  the true value $\Upsilon(S,\nu)=1/2$ (see Table \ref{tab1} for more results). 
Observe that in this case $\{x\in S: \Upsilon(S,\nu)=\lim_{r\to 0}\nu(\mathrm{B}(x,r))/\mu(\mathrm{B}(x,r))\}=\partial S$ which has $\nu$ measure 0. So, to ask 
\begin{equation}\label{hipaux}
	\nu\left(\left\{x\in S: \Upsilon(S,\nu)=\lim_{r\to 0}\frac{\nu(\mathrm{B}(x,r))}{\mu(\mathrm{B}(x,r))}\right\}\right)=0,
	\end{equation}
 does not guarantee a good performance for $\hat{\Upsilon}_n$. To address these issues, we will propose a correction for $\hat{\Upsilon}_n$ in the following section.
\end{remark}

%We suggest a correction for $\hat{\Upsilon}_n$ in the following subsection which, roughly speaking, takes into account  the number of sample points $X_i$ at which the minimum $\nu(\mathrm{B}(X_i,r_n))/\mu(\mathrm{B}(X_i,r_n))$ is attained.

\subsection{A bias-correction method}\label{biascorrect}

Let us consider the estimator $\hat{\Upsilon}_n$ introduced in \eqref{hatest0}.   Let $r_n=\left(\log(n)/n\right)^{1/(2d)}$.
We consider the following correction: 
\begin{equation*}\label{tildest0}
	\tilde{\Upsilon}_n\equiv \tilde{\Upsilon}_n(S,\nu)= \hat{\Upsilon}_n\left(1+\frac{{\#}\mathcal{A}_n}{n}\right),
\end{equation*}
where 
\begin{equation*}\label{setan}
	\mathcal{A}_n=\left\{i: \frac{\# (\aleph_n\cap \mathrm{B}(X_i,r_n))}{n\omega_dr_n^d}\leq \hat{\Upsilon}_n\Bigg(1+\omega_d r_n^{d/2}\Bigg)\right\}.
\end{equation*}
Then, $\#\mathcal{A}_n/n$ estimates the proportion of sample points $X_i$ such that $\#(\aleph_n\cap \mathrm{B}(X_{i},r_n))/(n\omega_dr_n^d)$ is `close' to the minimum $\hat{\Upsilon}_n$. Roughly speaking, this takes into account  the number of sample points $X_i$ at which the minimum $\nu(\mathrm{B}(X_i,r_n))/\mu(\mathrm{B}(X_i,r_n))$ is attained.

\begin{theorem}\label{sesgo} Under the hypotheses of Theorem \ref{th:esti}, and assuming  
also that \eqref{hipaux} is fulfilled, 
$$\tilde{\Upsilon}_n\to \Upsilon(S,\nu)\quad a.s.$$
\end{theorem}

\section{Simulations}\label{simus}

In this section, we conduct a small simulation study to show and compare the behaviour of estimators  $\hat{\Upsilon}_n(S,\nu)$ and $\tilde{\Upsilon}_n(S,\nu)$. We considered different situations by changing both the support $S$ and the dimension of the space, as well as the probability distribution $\nu$. Overall, the results confirm what was previously mentioned. The plug-in estimator generally does not perform well, especially in cases where $\min_{x\in S} \nu(\mathrm{B}(x,r))/\mu(\mathrm{B}(x,r))$ is attained at many points. In those cases, the estimator $\tilde{\Upsilon}_n(S,\nu)$ outperforms considerably  $\hat{\Upsilon}_n(S,\nu)$.

\subsection{Uniform samples on \texorpdfstring{$\mathbb{R}^2$}{R2}} \label{unifsampl}
The objective here is to compare the performance of $\hat{\Upsilon}_n(S,\nu)$ and $\tilde{\Upsilon}_n(S,\nu)$ when the sample is uniformly distributed. To achieve this, we considered various sets in $\mathbb{R}^2$, all of them with area one. Note that, with $\nu$ being the uniform distribution and $S$ having area one, it follows that $\Upsilon(S,\nu)\equiv \Upsilon(S,\mu)$. We examined the following sets: $S_1$ is an equilateral triangle, $S_2=[0,1]^2$ is the unit square, $S_3$ is a regular hexagon, and $S_4=\mathrm{B}(0,1/\sqrt{\pi})$.  It follows that $\Upsilon(S_1,\mu)=1/6$, $\Upsilon(S_2,\mu)=1/4$, $\Upsilon(S_3,\mu)=1/3$ and $\Upsilon(S_4,\mu)=1/2$.
 
Mean values of $\hat{\Upsilon}_n(S_i,\nu)$ and $\tilde{\Upsilon}_n(S_i,\nu)$ over 500 replications are shown in Table \ref{tab1}, for $i=1,\ldots, 4$. In brackets we show the variance of the estimators. In all cases $r_n=(\log(n)/n)^{1/4}$.

\begin{table}%[!htb]
	\begin{center}
	\caption{Mean values of $\hat{\Upsilon}_n(S_i,\mu)$ and $\tilde{\Upsilon}_n(S_i,\mu)$  over 500 replications  (variance shown in brackets). Samples are generated from the uniform distribution.}\label{tab1}
		\begin{tabular}{rccccc}
			\hline
			& $n= 1000$ & $n= 3000$ & $n= 5000$ & $n= 7000$ & $n= 9000$ \\ 
			\hline
			$\hat{\Upsilon}_n(S_1,\mu)$ & 0.1765 & 0.1758 & 0.1745 & 0.1730 & 0.1735 \\ 
			& (0.0005) & (0.0003) & (0.0002) & (0.0002) & (0.0001) \\ 
			$\tilde{\Upsilon}_n(S_1,\mu)$ & 0.1829 & 0.1782 & 0.1761 & 0.1742 & 0.1745 \\ 
			& (0.0007) & (0.0004) & (0.0002) & (0.0002) & (0.0001) \\ \hline
			$\hat{\Upsilon}_n(S_2,\mu)$ & 0.2510 & 0.2520 & 0.2527 & 0.2523 & 0.2521 \\ 
			& (0.0006) & (0.0003) & (0.0003) & (0.0002) & (0.0002) \\ 
			$\tilde{\Upsilon}_n(S_2,\mu)$ & 0.2767 & 0.2606 & 0.2581 & 0.2563 & 0.2552 \\ 
			& (0.0014) & (0.0004) & (0.0003) & (0.0002) & (0.0002) \\ \hline
			$\hat{\Upsilon}_n(S_3,\mu)$ & 0.3190 & 0.3236 & 0.3246 & 0.3251 & 0.3250 \\ 
			& (0.0005) & (0.0003) & (0.0002) & (0.0002) & (0.0002) \\ 
			$\tilde{\Upsilon}_n(S_3,\mu)$ & 0.4014 & 0.3650 & 0.3519 & 0.3453 & 0.3407 \\ 
			& (0.0023) & (0.0009) & (0.0006) & (0.0004) & (0.0003) \\ \hline
			$\hat{\Upsilon}_n(S,\mu)$ & 0.3813 & 0.4048 & 0.4131 & 0.4201 & 0.4231 \\ 
			& (0.0004) & (0.0002) & (0.0002) & (0.0001) & (0.0001) \\ 
			$\tilde{\Upsilon}_n(S_4,\mu)$ & 0.5380 & 0.5190 & 0.5088 & 0.5064 & 0.5015 \\ 
			& (0.0025) & (0.0008) & (0.0006) & (0.0005) & (0.0004) \\ \hline
			
			\hline
		\end{tabular}
		
	\end{center}
\end{table}

\subsection{Uniform samples on  \texorpdfstring{$\mathbb{R}^d$, for $d>2$}{Rd for d>2}} \label{unifsampl2}

Table \ref{tab1} illustrates that the estimator $\tilde{\Upsilon}_n$ significantly outperforms $\hat{\Upsilon}_n$. This performance gap increases with the dimension, as we will show here. We consider the uniform distribution $\nu$ on $S=\mathrm{B}(0,R)\subset \mathbb{R}^d$, where $R$ is chosen such that $\mu(S)=1$, for $d=3$ and $d=4$. In both cases, $\Upsilon(S,\nu)\equiv \Upsilon(S,\mu)=1/2$.
As before, we replicated the whole procedure 500 times and  report on Table \ref{tab11} the mean values of $\hat{\Upsilon}_n(S,\mu)$ and $\tilde{\Upsilon}_n(S,\mu)$ (variance between brackets). We have used $r_n=(\log(n)/n)^{1/(2d)}$.

\begin{table}%[!htb]
	\begin{center}
	\caption{Mean values of $\hat{\Upsilon}_n(S,\mu)$ and $\tilde{\Upsilon}_n(S,\mu)$ over 500 replications  (variance shown in brackets). For each value of $d$, the set $S$ is a $d$-ball $S=\mathrm{B}(0,R)\subset\mathbb{R}^d$ with $\mu(S)=1$. Samples are generated from the uniform distribution on $S$.}\label{tab11}
		\begin{tabular}{cccccc}
			\hline
			& $n= 1000$ & $n= 3000$ & $n= 5000$ & $n= 7000$ & $n= 9000$ \\ 
			\hline
			&   \multicolumn{5}{c}{$d=3$} \\ 
			\hline
			$\hat{\Upsilon}_n(S,\mu)$& 0.3110 & 0.3385 & 0.3496 & 0.3579 & 0.3634 \\ 
			& (0.00021) & (0.00011) & (0.00009) & (0.00008) & (0.00007) \\ 
			$\tilde{\Upsilon}_n(S,\mu)$ & 0.5209 & 0.5203 & 0.5151 & 0.5139 & 0.5119 \\ 
			& (0.00150) & (0.00072) & (0.00055) & (0.00044) & (0.00040) \\ 
			
			\hline
			&   \multicolumn{5}{c}{$d=4$} \\ 
			\hline
			$\hat{\Upsilon}_n(S,\mu)$& 0.2485 & 0.2766 & 0.2897 & 0.2982 & 0.3043 \\ 
			& (0.00013) & (0.00007) & (0.00006) & (0.00005) & (0.00004) \\ 
			$\tilde{\Upsilon}_n(S,\mu)$ & 0.4438 & 0.4620 & 0.4677 & 0.4712 & 0.4728 \\ 
			&(0.00087) & (0.00047) & (0.00037) & (0.00033) & (0.00026)\\

			\hline
		\end{tabular}
\end{center}
\end{table}

 \subsection{Non-uniform samples}
 
For the non-uniform case, we consider $\nu$ to be a mixture of distributions. Specifically, let $\theta_1$ and $\theta_2$ be uniform distributions on the interval $[0,2\pi]$, and let $R_1$ and $R_2$ be uniform distributions on the interval $[0,1]$. We define $Y_1= R_1^{1/2} \left( \cos (\theta_1), \sin(\theta_1) \right)$, which is uniformly distributed in the unit ball and $Y_2= R_2^{1/4} \left( \cos (\theta_2), \sin(\theta_2) \right)$, whose distribution has an `antimode' at the origin $(0,0)$. We consider the mixture $Y=\frac{1}{4} (Y_1 + 3Y_2)$, with  support $S=\mathrm{B}(0, 1)$.
 
  In this case the standardness constant is $\Upsilon(S,\nu)=1/(4\pi)\approx 0.07958$. Mean values of $\hat{\Upsilon}_n(S,\nu)$ and $\tilde{\Upsilon}_n(S,\nu)$ over 500 replications are shown in Table \ref{tab2}. In brackets we show the variance of the estimators. In all cases $r_n=(\log(n)/n)^{1/4}$.

\begin{table}%[!htb]
\begin{center}
	\caption{Mean values of $\hat{\Upsilon}_n(S,\mu)$ and $\tilde{\Upsilon}_n(S,\mu)$ over 500 replications  (variance shown in brackets). Samples are generated from a non-uniform distribution on $S=\mathrm{B}(0,1)\subset\mathbb{R}^2$.}\label{tab2}
		\begin{tabular}{cccccc}
			\hline
			& $n= 1000$ & $n= 3000$ & $n= 5000$ & $n= 7000$ & $n= 9000$   \\ \hline
			$\hat{\Upsilon}_n(S,\mu)$ & 0.0857 & 0.0789 & 0.0775 & 0.0768 & 0.0771 \\ 
			& (0.000212) & (0.000109) & (0.000073) & (0.000064) & (0.000060) \\ 
			$\tilde{\Upsilon}_n(S,\mu)$& 0.0924 & 0.0817 & 0.0798 & 0.0787 & 0.0789 \\ 
			& (0.000399) & (0.000145) & (0.000093) & (0.000081) & (0.000074) \\ 
			\hline
			\end{tabular}
\end{center}
\end{table}

\section{Non-decidability}  \label{noesti}

Our theoretical analysis, presented in Section \ref{sec:est_stand}, demonstrates that the standardness constant can be estimated consistently. Nevertheless, we prove in this section that it is not possible to decide, with a given confidence, whether or not the support of a probability measure $\nu$, ${\textnormal{supp}}(\nu)$, is standard with respect to $\nu$, from a sample. As mentioned in the Introduction, this problem aligns with a broader statistical challenge concerning the impossibility of consistently estimating certain functionals of probability measures from finite samples. We examine how this general problem manifests in our context. %the set is standard, is not possible. 
More precisely, if  $\mathcal{M}$ denotes the (separable and complete) metric space of Borel probability distributions on $\mathbb{R}^d$ endowed with the total variation norm $\|\cdot\|_{TV}$, we prove that the functionals defined on $\mathcal{M}$ as  
 
\begin{equation}\label{densi}
\alpha(\nu)=
\begin{cases}
1 & \quad \text{ if } \nu \ll \mu,\\
0 & \quad \text{ otherwise,}
\end{cases}
\end{equation}
\begin{equation}\label{clausint}
	\beta(\nu)=
	\begin{cases}
		1 & \quad \text{ if } {\textnormal{supp}}(\nu)=\overline{{\textnormal{int}}({\textnormal{supp}}(\nu)),}\\
		0 & \quad \text{ otherwise}
	\end{cases}
\end{equation}
and 
\begin{equation}\label{st}
	\gamma(\nu)=
	\begin{cases}
		1 & \quad \text{ if } {\textnormal{supp}}(\nu) \text{ is standard  with respect to} \ \nu,\\
		0 & \quad \text{ otherwise,}
	\end{cases}
\end{equation}
cannot be consistently estimated from a sequence of i.i.d. observations from a random variable distributed as $\nu$.  To this aim we prove the following Lemma, whose proof is similar to that of Lemma 1.1 in \cite{fm99}.

\begin{lemma} \label{lemaux} Let $\mathcal{B}\subset \mathcal{M}$ be equipped with some norm $\|\cdot\|$ such that $(\mathcal{B},\|\cdot \|)$ is a complete metric space. Assume that $\|\nu\|_{TV}\leq c\|\nu \|$ for some constant $c$ and all $\nu\in \mathcal{B}$. Let $\phi:\mathcal{B}\to [-K,K]$
be any bounded characteristic of the distributions in $\mathcal{B}$. If $\phi$ is consistently estimable on $\mathcal{B}$, then there exists a dense subset of points in $\mathcal{B}$ at which  $\phi$ is continuous with respect to the topology induced by $\|\cdot\|$. Consequently, if $\phi$ is discontinuous at every point in $\mathcal{B}$, then it is not consistently estimable on $\mathcal{B}$.
\end{lemma}

\begin{theorem}\label{thdensity}  
The functional $\alpha$, defined on $\mathcal{M}$ by \eqref{densi}, cannot be consistently estimated from a sequence of i.i.d. observations.
\end{theorem}

\begin{theorem}  \label{thstnoesti}  The functionals $\beta$ and $\gamma$, defined on $\mathcal{M}$ by \eqref{clausint} and \eqref{st}, 	cannot be consistently estimated from a sequence of i.i.d. observations. 
\end{theorem}

\section{Concluding remarks}\label{sec:conclusion}

The standardness constant, while not extensively examined in the literature, is essential for set estimation. It imposes constraints on both the underlying distribution and the shape of its support, influencing the convergence rates of various well-known set estimators, including the Devroye-Wise estimator. The naive estimator, $\hat{\Upsilon}_n$, has received limited attention; existing studies have only established convergence in probability without specifying rates, often under much stronger assumptions than those we consider in Theorem \ref{th:esti}

Our simulations reveal that, despite its intuitive appeal, the naive estimator performs poorly in simple scenarios, such as when the distribution is uniform on a given set. Its performance further deteriorates as the dimensionality increases. To mitigate this issue, we propose a bias-correction method. Our simulations demonstrate that this method significantly outperforms the naive estimator, offering a more robust alternative in higher dimensions and diverse distributional contexts.

Determining whether this method achieves the optimal rate of convergence is beyond the scope of this paper. Additionally, we establish that it is impossible to determine, from a finite sample, whether the support is standard.

%%%%%%%%%%%%%%%%%%%%%%%%%%%%%%%%%%%%%%%%%%%%%%
%% Example with single Appendix:            %%
%%%%%%%%%%%%%%%%%%%%%%%%%%%%%%%%%%%%%%%%%%%%%%
\begin{appendix}
\section{Proofs of the theoretical results in Sections \ref{sec-back} and \ref{sec:est_stand}}
%In this section we provide the proofs of the theoretical results  presented in Sections \ref{sec-back} and \ref{sec:est_stand} .\\

\begin{proof}[Proof of Lemma \ref{lemaleo}]

Let us denote $\delta_{\text{est}}=  \lim_{r \rightarrow 0} \inf_{x \in S} \psi(x,r)$, where 
	$$\psi(x,r)= \nu(\mathrm{B}(x,r)\cap S)/(r^d\omega_d).$$ 
	Let 	$A$ be the set of all pairs $(\delta, \lambda)$ fulfilling \eqref{estandar}, and $\Upsilon(S,\nu)= \sup_{(\delta, \lambda) \in A} \delta$. From {\textnormal{(H1)}} it follows that, for $r$ small enough, $A$ is non-empty. We have to prove that $\delta_{\text{est}}=\Upsilon(S,\nu)$. Let us start by proving that $\Upsilon(S,\nu) \leq \delta_{\text{est}}$.  Let us fix $(\delta, \lambda)\in A$.     Then 
	$\psi(x,r) \geq \delta$  for all  $x\in S$ and $0<r<\lambda$. Then, we have
	$$\delta_{\text{est}}=  \lim_{r \rightarrow 0} \inf_{x \in S} \psi(x,r) \geq \delta.$$
	Therefore, it follows that $\Upsilon(S,\nu) \leq \delta_{\text{est}}$, because $\Upsilon(S,\nu)$ is the smallest upper bound.
	
	To prove  $\Upsilon(S,\nu) \geq \delta_{\text{est}}$, let us assume that $\Upsilon(S,\nu) < \delta_{\text{est}}$. Define $\gamma^*=(\delta_{\text{est}} -\Upsilon(S,\nu))/2$. There exists  $r^*>0$ such that, for all $0<r<r^*$,
	
	$$\inf_{x \in S} \psi(x,r) \in \mathrm{B}(\delta_{\text{est}}, \gamma^* ).$$
	Then, for all  $x \in S$ and  $0<r<r^*$,
	$$\Upsilon(S,\nu) < \delta^*  < \inf_{x \in S} \psi(x,r) \leq \psi(x,r),$$
	where  $\delta^* = \delta_{\text{est}}  - \gamma^*$. This implies that  $(\delta^*,  r^*) \in A$ and $\delta^*> \Upsilon(S,\nu)$ which is a contradiction.
	\end{proof}

\begin{proof}[Proof of Proposition \ref{lemgeo}]

	Let us first prove  that \eqref{prophip} implies that $S$ is standard with respect to $\mu$. Let $0<\eps<\mathcal{L}_{\partial S}/4$, and take $r_0$ such that $|\mathcal{L}_{\partial S}-\min_{x\in \partial S}\mu(\mathrm{B}(x,r)\cap S)/(r^d\omega_d)|<\eps$, for all $r< r_0$. For all $x\in S$ such that $d(x,\partial S)\geq r_0$, 
	$$\frac{\mu(\mathrm{B}(x,2r)\cap S)}{(2r)^d\omega_d}\geq \frac{1}{2^d}\frac{\mu(\mathrm{B}(x,r)\cap S)}{r^d\omega_d}=\frac{1}{2^d}.$$
	If $d(x,\partial S)\leq r_0$, let $y$ the closest point to $x$ in $\partial S$, then
	$$\frac{\mu(\mathrm{B}(x,2r)\cap S)}{(2r)^d\omega_d}\geq \frac{1}{2^d}\frac{\mu(\mathrm{B}(y,r)\cap S)}{r^d\omega_d}\geq\frac{1}{2^d}\min_{y\in \partial S} \frac{\mu(\mathrm{B}(y,r)\cap S)}{r^d\omega_d}\geq \frac{3\mathcal{L}_{\partial S}}{2^{d+2}}.$$
	Since $f$ is bounded from below by a positive constant, and $S$ is standard with respect to $\mu$, it follows easily that $S$ is standard with respect to $\nu$. 
	
	Let us prove that $S$ fulfils condition {\textnormal{(H1)}}. From Lemma \ref{lemaleo}, we have to prove that  $\lim_{r\to 0} \min_{x\in S} \psi(x,r)$ exists and is given by \eqref{estdens}. Let $x_0\in\text{argmin}_{x\in S} f(x)$.
	Since $S$ is compact and $f$ is continuous, it is uniformly continuous, that is, for all $\eps>0$ there exists $\delta>0$ such that $|f(x)-f(y)|\leq \eps$ for all $\|x-y\|<\delta$. 
	Let $\eps>0$ and $r<\delta/2$. Then,
	$$\psi(x_0,r):=\frac{\int_{\mathrm{B}(x_0,r)\cap S}f(t)dt}{r^d\omega_d}\leq (f(x_0)+\eps) \frac{\mu(\mathrm{B}(x_0,r)\cap S)}{r^d\omega_d}\leq f(x_0)+\eps.$$
	Then, it follows that  $\limsup_{r\to 0} \min_{x\in S} \psi(x,r)\leq \min_{t\in S}f(t)$. 
	Let $r<\delta/2$,
	$$\min_{x\in S} \psi(x,r)\leq \min_{x\in \partial S}\frac{\int_{\mathrm{B}(x,r)\cap S}f(t)dt}{r^d\omega_d}\leq \eps+\min_{x\in \partial S} f(x)\frac{\mu(\mathrm{B}(x,r)\cap S)}{r^d\omega_d}.$$
	Then,
	\begin{equation*}
	\limsup_{r\to 0}\min_{x\in S} \psi(x,r)\leq \min\Bigg\{\min_{x\in S}f(x), \lim_{r\to 0} \min_{x\in \partial S} f(x)\frac{\mu(\mathrm{B}(x,r)\cap S)}{r^d\omega_d}\Bigg\}.\end{equation*}
	
	It remains to be proved that $\liminf_{r\to 0}\min_{x\in S} \psi(x,r)$ is bounded from below by the right-hand side of \eqref{estdens}. 
	Since $\psi(x,r)$ is a continuous function of $x$ we can choose $r_n\to 0$ and $x_n$ such that 
	
\begin{equation*}
\psi(x_n,r_n)=\min_{x\in S} \psi(x,r_n)\quad\text{ and }\quad \psi(x_n,r_n)\to \liminf_{r\to 0}\min_{x\in S} \psi(x,r).
\end{equation*}
	
	We can assume, by considering a subsequence if necessary, that $x_n\to x^*\in S$. If $x^*\in {\textnormal{int}}(S)$, then there exists $\eta>0$ such that $\mathrm{B}(x^*,\eta)\subset S$. Let $\eta^*=\min\{\eta/2,\delta/2\}$, being $\delta$ from the uniform continuity of $f$. Then, for all $n$ large enough such that $r_n<\eta^*$, $\psi(x_n,r_n)>f(x^*)+\eps$, from where it follows that $\liminf_{r\to 0}\min_{x\in S} \psi(x,r)\geq f(x^*)\geq \min_{x\in S} f(x)$. If $x^*\in \partial S$, for all $n$ large enough such that $r_n<\eta^*$,

	\begin{align*}
		\psi(x_n,r_n)&\geq (f(x^*)-\eps)\frac{\mu(\mathrm{B}(x^*,r_n)\cap S)}{\mu(\mathrm{B}(x^*,r_n))}\geq f(x^*)\frac{\mu(\mathrm{B}(x^*,r_n)\cap S)}{\mu(\mathrm{B}(x^*,r_n))}-\eps\\
		& \geq  \min_{x\in \partial S} f(x)\frac{\mu(\mathrm{B}(x,r_n)\cap S)}{\mu(\mathrm{B}(x,r_n))}-\eps,
	\end{align*}
	from where it follows that 
	$$\liminf_{r\to 0}\min_{x\in S} \psi(x,r)\geq  \lim_{r\to 0} \min_{x\in \partial S}  f(x) \frac{\mu(\mathrm{B}(x,r)\cap S)}{r^d\omega_d},$$
	which concludes the proof.
\end{proof}

\begin{proof}[Proof of Theorem \ref{th:esti}]

		From Corollary 13.2 in \cite{luc}, $\mathcal{C}=\{\mathrm{B}(x,r_n):x\in S, n \in \mathbb{N}\}$ has  Vapnik-Chervonenkis  dimension at most $d+2$. Then, from Theorems 12.5 and 13.3 in \cite{luc}, 
		$$\mathbb{P}\left(\sup_{B=\mathrm{B}(x,r_n)\in \mathcal{C}}\left|\frac{\#\{\aleph_n\cap B\}}{n}-\nu(B)\right|>\eps \frac{r_n^d}{\beta_n}\omega_d\right)\leq 8n^{d+2}\exp\Big(\frac{-n\eps^2r_n^{2d}\omega_d^2}{32\beta_n^2}\Big).$$
		Furthermore,  from Borel Cantelli's lemma, 
		\begin{equation}\label{theq1}
			\beta_n\sup_{B=\mathrm{B}(x,r_n)\in \mathcal{C}}\frac{1}{r_n^d\omega_d}\left|\frac{\#\{\aleph_n\cap B\}}{n}-\nu(B)\right|\to 0\quad \text{ a.s}.
		\end{equation}
	Let $x_{r}$ denote a point in $S$ that minimizes $\nu(\mathrm{B}(x,r))/(r^d\omega_d)$, and $X_{i,r_n}$ the closest point in $\aleph_n$ to $x_{r_n}$. Then, with probability one, for $n$ large enough,  
	 $$\|X_{i,r_n}-x_{r_n}\| \leq d_H(\aleph_n,S)\leq K(\log(n)/n)^{1/d},$$

		for some $K>0$, (see Theorem 3 in  \cite{cuevas2004}). Let us denote $\eps_n=K(\log(n)/n)^{1/d}$. 
		Observe that $\eps_n/r_n\to 0$ since by hypothesis $(\eps_n/r_n)(\beta_n^2/r_n^{2d-1})\to 0$  and 
		$\beta_n^{2} / r_n^{2d-1}$ to $\infty$.

 Then, with probability one, for $n$ large enough,
		\begin{align}\label{cotath1}
			\beta_n\Bigg|\frac{\nu(\mathrm{B}(X_{i,r_n},r_n))}{r_n^d\omega_d}-\frac{\nu(\mathrm{B}(x_{r_n},r_n))}{r_n^d\omega_d}\Bigg|&\leq \beta_n\frac{\nu(\mathrm{B}(x_{r_n},r_n+\eps_n))}{r_n^d\omega_d} -\beta_n\frac{\nu(\mathrm{B}(x_{r_n},r_n))}{r_n^d\omega_d}\\\nonumber
			&\leq \beta_n\Bigg|\frac{\nu(\mathrm{B}(x_{r_n},r_n+\eps_n))}{r_n^d\omega_d}-\Upsilon(S,\nu)\Bigg|+\beta_n\Omega(r_n).\nonumber
		\end{align}
Since 
		$$\frac{\nu(\mathrm{B}(x_{r_n},r_n+\eps_n))}{r_n^d\omega_d}=\frac{\nu(\mathrm{B}(x_{r_n},r_n+\eps_n))}{(r_n+\eps_n)^d\omega_d}\frac{(r_n+\eps_n)^d}{r_n^d},$$
it follows that
\begin{equation*}
			\Bigg|\frac{\nu(\mathrm{B}(x_{r_n},r_n+\eps_n))}{r_n^d\omega_d}-\Upsilon(S,\nu)\Bigg|\leq \Omega(r_n+\eps_n)\Big(1+\frac{\eps_n}{r_n}\Big)^d+
			\Upsilon(S,\nu)\Bigg|1-\Big(1+\frac{\eps_n}{r_n}\Big)^d\Bigg|.
		\end{equation*}	 
If use that $(1+x)^d=1+dx+o(x^2)$, then, for $n$ large enough,
		\begin{equation*}
			\Bigg|\frac{\nu(\mathrm{B}(x_{r_n},r_n+\eps_n))}{r_n^d\omega_d}-\Upsilon(S,\nu)\Bigg|\leq 2\Omega(r_n+\eps_n)+\Upsilon(S,\nu)d\frac{\eps_n}{r_n}+o\Big(\frac{\eps_n^2}{r_n^2}\Big).
		\end{equation*}	
Since $\beta_n\eps_n/r_n\to 0$, and $\beta_n\Omega(r_n^{1+\gamma})\to 0$,
		
		\begin{equation}\label{cotath2}
			\beta_n\Bigg|\frac{\nu(\mathrm{B}(X_{i,r_n},r_n))}{r_n^d\omega_d}-\frac{\nu(\mathrm{B}(x_{r_n},r_n))}{r_n^d\omega_d}\Bigg|\to 0.
		\end{equation}
		From \eqref{cotath2}, \eqref{theq1}, and $\beta_n\Omega(r_n)\to 0$,
		$$\beta_n\Bigg|\frac{\#\{\aleph_n\cap \mathrm{B}(X_{i,r_n},r_n)\}/n}{r_n^d\omega_d}-\Upsilon(S,\nu)\Bigg|\to 0.$$
		Then
		$$\beta_n(\hat{\Upsilon}_n-\Upsilon(S,\nu))\leq \beta_n\Bigg|\frac{\#\{\aleph_n\cap \mathrm{B}(X_{i,r_n},r_n)\}/n}{r_n^d\omega_d}-\Upsilon(S,\nu)\Bigg|\to 0.$$
		This proves that, with probability one, for all $n$ such that $\beta_n(\hat{\Upsilon}_n-\Upsilon(S,\nu))\geq 0$, $\beta_n(\hat{\Upsilon}_n-\Upsilon(S,\nu))\to 0$ a.s. Let us define $\mathcal{N}=\{n\in \mathbb{N}: \beta_n(\hat{\Upsilon}_n-\Upsilon(S,\nu))<0\}$. If $\#\mathcal{N}<\infty$, then \eqref{limite} follows. Let us consider the case $\#\mathcal{N}=\infty$. Assume by contradiction that there exists $\gamma>0$ such that, with positive probability, 
		$$\beta_n(\Upsilon(S,\nu)-\hat{\Upsilon}_n)>\gamma\quad \text{ for all } n>n_0, n\in \mathcal{N}.$$
		For all $n\in \mathcal{N}$ and $n>n_0$ let  $X_{j,r_n}$ a sample point that minimize $\#\{\aleph_n\cap \mathrm{B}(X_i,r_n)\}/(r_n^2\omega_d)$. From \eqref{theq1}, for $n>n_1\geq n_0$
		$$\beta_n\Bigg(\Upsilon(S,\nu)-\frac{\nu(\mathrm{B}(X_{j,r_n},r_n)\cap S)}{r_n^d\omega_d}\Bigg)>\gamma/2>0.$$
Now observe that 
\begin{align*}
			\beta_n\Bigg(\Upsilon(S,\nu)-\min_{x\in S}\frac{\nu(\mathrm{B}(x,r_n)\cap S)}{r_n^d\omega_d}\Bigg)&\geq  \beta_n\Bigg(\Upsilon(S,\nu)-\frac{\nu(\mathrm{B}(X_{j,r_n},r_n)\cap S)}{r_n^d\omega_d}\Bigg)\\
			&>\gamma/2>0
		\end{align*}
		which is a contradiction because the left-hand side converges to 0.
	\end{proof} 

\begin{proof}[Proof of Proposition \ref{omega}] 
Assume that $\mu(S)=1$ for simplicity. We know that $\Upsilon(S,\nu)=\lim_{r\to 0} \nu(\mathrm{B}(x_r,r))/(\omega_dr^d)$. Then, 
 \begin{equation*} 
\Bigg |\Upsilon(S,\nu)- \lim_{r\to 0}   f(x_r) \frac{\mu(\mathrm{B}(x_r,r)\cap S)}{\omega_dr^d} \Bigg| \leq  \lim_{r\to 0}   \frac{\int_{\mathrm{B}(x_r,r)\cap S} |f(t)-f(x_r)|dt}{\omega_dr^d}=0.
 \end{equation*} 
Hence,
\begin{equation}\label{eq0}
	\Upsilon(S,\nu)=\lim_{r\to 0}   f(x_r) \frac{\mu(\mathrm{B}(x_r,r)\cap S)}{\omega_dr^d}.
\end{equation} 
From 
 $$\Bigg|\frac{\nu(\mathrm{B}(x_r,r)\cap S)}{\mu(\mathrm{B}(x_r,r))}-f(x_r)\frac{\mu(\mathrm{B}(x_r,r)\cap S)}{\omega_dr^d}\Bigg|\leq  \frac{\int_{\mathrm{B}(x_r,r)\cap S} |f(t)-f(x_r)|dt}{\omega_dr^d}\leq Lr^\alpha$$ 
and \eqref{eq0}, it follows 
 $$\Omega(r)\leq  \Bigg|f(x^*)\lim_{r\to 0}  \frac{\mu(\mathrm{B}(x_r,r)\cap S)}{\omega_dr^d} -f(x_r)\frac{\mu(\mathrm{B}(x_r,r)\cap S)}{\omega_dr^d}\Bigg|+Lr^\alpha.$$
If $x^*\in \textnormal{int}(S)$, and since $x^*$ is a local minimum, we have that  $\frac{\mu(\mathrm{B}(x_r,r)\cap S)}{\omega_dr^d}=1$ for $r$ small enough. Then
 $\Omega(r)=0$ for all $r$ small enough.  Assume now that $x^*\in \partial S$. By Corollary  \ref{coro2},
\begin{equation}\label{eq00}
\lim_{r\to 0}  \frac{\mu(\mathrm{B}(x_r,r)\cap S)}{\omega_dr^d} =1/2.
\end{equation}
Moreover, by the rolling conditions
\begin{equation}\label{orden2}
	\Big|\frac{\mu(\mathrm{B}(x^*,r)\cap S)}{\omega_dr^d}-\frac{1}{2}\Big|=O(r^{2d+1}).
\end{equation}
Then, 
  $$\Omega(r)\leq \Bigg|\frac{f(x_r)}{2}-\frac{f(x^*)}{2}\Bigg|+O(r^{2d+1})+Lr^\alpha\leq L|x_r-x^*|^\alpha+O(r^{2d+1}).$$
Observe that \eqref{eq00} implies that $|x_r-x^*|<r$.
 \end{proof}
 
 \begin{proof}[Proof of Theorem \ref{sesgo}]

		To prove that $\tilde{\Upsilon}_n\to \Upsilon(S,\nu)\   a.s.$ it is enough to prove that $\#\mathcal{A}_n/n\to 0$.

		Let $\beta_n$ as in Theorem \ref{th:esti}. For all $\eps>0$, with probability one, for all $n$ large enough, $\hat{\Upsilon}_n\leq \Upsilon(S,\nu)+\eps/\beta_n$.
		Then there exists $C$ a positive constant such that, with probability one, for all $n$ large enough,
		$$\#\mathcal{A}_n\leq \#\left\{i: \frac{\# (\aleph_n\cap \mathrm{B}(X_i,r_n))}{n\omega_dr_n^d}\leq \Upsilon(S,\nu)+\frac{\eps}{\beta_n}+Cr_n^{d/2}\right\}:=\#\mathcal{B}_n.$$
		From \eqref{theq1}, with probability one, for all $n$ large enough,
		$$\#\mathcal{B}_n\leq \#\left\{i: \frac{\nu(\mathrm{B}(X_i,r_n))}{\omega_dr_n^d}\leq \Upsilon(S,\nu)+2\frac{\eps}{\beta_n}+Cr_n^{d/2}\right\}.$$
		Let denote
		$$\mathbb{C}_n=\left\{x\in S: \frac{\nu(\mathrm{B}(x,r_n))}{\omega_dr_n^d}\leq \Upsilon(S,\nu)+2\frac{\eps}{\beta_n}+Cr_n^{d/2}\right\},$$
		and
		$$\mathbb{C}=\left\{x\in S:\Upsilon(S,\nu)=\lim_{r\to 0}  \frac{\nu(\mathrm{B}(x,r))}{\omega_dr^d}    \right\},$$
		then, with probability one, for $n$ large enough, $\#\mathcal{A}_n/n\leq (1/n)\sum_{i=1}^n \mathbb{I}_{\mathbb{C}_n}(X_i)$. From Hoeffding' inequality it follows that 
		$$\Bigg|\frac{1}{n}\sum_{i=1}^n \mathbb{I}_{\mathbb{C}_n}(X_i)-\nu(\mathbb{C}_n)\Bigg|\to 0, \quad a.s.$$
		Observe that  $\mathbb{I}_{\mathbb{C}_n}(y)= 0$  for $n$ large enough, for all $y\in \mathbb{C}^c$, from where it follows that $\nu(\mathbb{C}_n)\to 0$, which concludes the proof.
	\end{proof}

 \section{Proofs of the theoretical results  in Section \ref{noesti}}
%In this section we provide the proofs of the theoretical results  presented in Section \ref{noesti}.

		\begin{proof}[Proof of Lemma \ref{lemaux}]

Suppose that $\phi$ is estimable in $\mathcal{B}$, from $X_1,\dots,X_n$ i.i.d. with distribution $\nu$, with the estimator $T_n(X_1,\dots,X_n)$. Let us define
			$$S_n(X_1,\dots,X_n)=T_n\mathbb{I}_{|T_n|\leq K}+ \text{sign}(T_n)K\mathbb{I}_{|T_n|>K}.$$
			As in \cite{fm99}, $S_n$ is a consistent estimate of $\phi$ on $\mathcal{B}$. Define $\phi_n:\mathcal{B}\to [-K,K]$ by 
			\begin{align*}
				\phi_n(\nu)=\E_\nu (S_n)&=\int\dots \int S_n(x_1,\dots,x_n)d\nu\dots d\nu\\
				&= \int_{(\mathbb{R}^d)^n} S_n(x_1,\dots,x_n)d\nu^n(x_1,\dots,x_n),
			\end{align*}
			where $\nu^n$ is the product measure of $\nu$, $n$ times.
			Since $S_n$ is bounded, then $\phi_n(\nu)\to \phi(\nu)$, for all $\nu \in \mathcal{B}$.
			In addition, for all $\nu,\mu\in \mathcal{B}$,
			\begin{align*}
				|\phi_n(\nu)-\phi_n(\mu)|&=  |\E_\nu(S_n)-\E_\mu(S_n)|\leq K\|\nu^n-\mu^n\|_{TV}\leq Kn\|\nu-\mu\|_{TV}\\
				&\leq ncK\|\nu-\mu\|.
			\end{align*}
			Now as in \cite{fm99}, using Baire's theorem,  it follows that $\phi$ must be continuous on a dense subset of $\mathcal{B}$.
		\end{proof}

\begin{proof}[Proof of Theorem \ref{thdensity}]
From Lemma \ref{lemaux} it is enough to prove that $\alpha$ is discontinuous at every $\nu\in \M$. Let us consider $\nu\in \M$ such that $\nu \ll \mu$. For all $\eps>0$  there exists $x$ and $r>0$ such that $\nu(\mathrm{B}(x,r))<\eps/2$, where it can be $\nu(\mathrm{B}(x,r))=0$. Let $z,y\in \mathrm{B}(x,r)$  such that $z\neq y$. Let $U$ be a random variable with uniform distribution on the segment $[z,y]$ joining $z$ and $y$. Let us consider $B$ a Bernoulli random variable with $\mathbb{P}(B=1)=\eps/2$. Let us consider a random vector  $X$ such that if $B=1$ then $X$ is distributed as the restriction of $\nu$ to ${\textnormal{supp}}(\nu)\setminus \mathrm{B}(x,r)$, on the contrary if $B=0$, $X=U$. Let $\lambda$ be the distribution of $X$, then $\lambda$ is not absolutely continuous with respect to the Lebesgue measure $\mu$, and $\|\lambda-\nu\|_{TV}<\eps$.
		
		Let us consider now $\nu$ such that $\alpha(\nu)=0$. Let  $\lambda_n$ be the probability measure on $\mathbb{R}^d$  associated to a Gaussian kernel $K_n$, with mean $0$ and variance $(1/n)I_d$ being $I_d$ the identity matrix in $\mathbb{R}^d$. Let $\nu\ast \lambda_n$ be the convolution of $\nu$ and $\lambda_n$, then it has density $K_n\ast \nu$ defined by 
		\begin{equation}\label{convoden}
			K_n\ast \nu(y)= \int K_n(y-x)d\nu(x).
		\end{equation}
		Since  $\lambda_n\to \delta_0$ the point mass in 0, and $\nu\ast \delta_0=\nu$, from \cite{part85}, $\lambda_n\ast \nu\to \nu$ in total variation, which concludes the proof of the theorem.
	\end{proof}
		
\begin{proof}[Proof of Theorem \ref{thstnoesti}]
 Let us start by proving that \eqref{clausint} cannot be consistently estimated. We will apply Lemma \ref{lemaux}. Let $\nu \in \M$ be such that $\beta(\nu)=0$. Let us denote $S={\textnormal{supp}}(\nu)$. Let us consider the sequence of sets $S_n=\mathrm{B}(S,1/n)$. Then $\overline{{\textnormal{int}}(S_n)}=S_n$. Since $S$ is closed, $\cap_n S_n=S$ . Let us consider $\lambda$ the standard Gaussian measure on $\mathbb{R}^d$. For all $\eps>0$ there exists $n$ large enough such that $\lambda(S_n\setminus S)<\eps$. The measure $\xi$ defined as 
		$$\xi(A)=\frac{\nu(A\cap S)+\lambda(A\cap (S_n\setminus S))}{1+\lambda(S_n\setminus S)},$$
		where $A$ is a Borel set, is a probability measure supported on $S_n$, $\beta(\xi)=1$ and $\|\xi-\nu\|_{TV}<2\eps$. Let $S$ such that $\beta(S)=1$. We can proceed as in the proof of Theorem \ref{thdensity}; let $\eps>0$, we first consider $\mathrm{B}(x,r)$ such that $\nu(\mathrm{B}(x,r))<\eps$ (it can be $\nu(\mathrm{B}(x,r))=0$), then we consider a non-empty segment on $\mathrm{B}(x,r)$ define the uniform on that segment. The measure $\lambda$ defined as in Theorem \ref{thdensity} fulfils that $\beta(\lambda)=0$ and $\|\lambda-\nu\|_{TV}<\eps$.
		
		To prove that \eqref{st} cannot be consistently estimate let us consider first $\nu\in \M$ such that $\gamma(\nu)=0$.  Let $\eps>0$ and $R$ large enough such that $\nu(\mathrm{B}(0,R)^c)<\eps$. 
		Let $S_\eps=S\cap \mathrm{B}(0,R)$ and $S_n=\mathrm{B}(S_\eps,1/n)$. Observe that $S_n$ is standard with respect to the $\mu$. Reasoning as before, $S_\eps$ is closed and then $\cap_n S_n=S$. Let us consider $\nu\ast \lambda_n$ the convolution of $\nu$ with $\lambda_n$, the Gaussian measure on $\mathbb{R}^d$ with variance $1/n$, as in Lemma \ref{lemaux}. There exists $n$ large enough such that $\nu*\lambda_n(S_n\setminus S_\eps)<\eps$. The  density of  $\nu*\lambda_n$  is given by \eqref{convoden}. Let $M$ be a compact set such that $\nu(M)>0$, then for all $z\in S_n$ and all $r>0$,
		\begin{align*}%\label{cota}
			\nu*\lambda_n(\mathrm{B}(z,r)\cap S_n)&=\int_{\mathrm{B}(z,r)\cap S_n} \int_{\mathbb{R}^d}  K_n(y-x)d\nu(x) dy\geq 			\int_{\mathrm{B}(z,r)\cap S_n }  K_n(y-x)d\nu(x) dy.
		\end{align*}
		If $z\in S_n$ and $y\in \mathrm{B}(z,r)\cap S_n$ then $y\in \mathrm{B}(S_n,r)$ which is a compact set, then there exists $c>0$ such that $K_n(y-x)\geq c$ for all $x\in M$ and $y\in \mathrm{B}(z,r)\cap S_n$, then
		\begin{equation*}%\label{cota}
			\nu*\lambda_n(\mathrm{B}(z,r)\cap S_n)\geq c\nu(M)\mu(\mathrm{B}(z,r)\cap S_n).
		\end{equation*}
		Then  $\gamma((\nu\ast\lambda_{n})_{|S_n})=1$ where $(\nu\ast \lambda_{n})_{|S_n}$ is the restriction of $\nu\ast\lambda_n$ to $S_n$.  Finally, from $(\nu\ast \lambda_{n})_{|S_n}\to \nu_{|S_\eps}$ in total variation, it follows that, for $n$ large enough, $\|\nu-\nu\ast \lambda_{n_{|S_n}}\|_{TV}<2\eps$.  It remains to be proved that if  $\nu$ is such that $\gamma(\nu)=1$, then for all $\eps>0$ there exists $\nu_\eps$ such that $\|\nu_\eps-\nu\|_{TV}<\eps$ and $\gamma(\nu_\eps)=0$. We proceed as in the proof that \eqref{clausint} cannot be estimated. We first consider $\mathrm{B}(x,r)$ such that $\nu(\mathrm{B}(x,r))<\eps$ (it can be $\nu(\mathrm{B}(x,r))=0$), then we consider a non-empty segment on the ball and define the uniform on that segment. The measure $\lambda$ defined as in Theorem \ref{thdensity} fulfils that $\gamma(\lambda)=0$ and $\|\lambda-\nu\|_{TV}<\eps$.
	\end{proof}
\end{appendix}

%%%%%%%%%%%%%%%%%%%%%%%%%%%%%%%%%%%%%%%%%%%%%%
%% Acknowledgements                         %%
%% should be provided in the                %%
%% Acknowledgements section.                %%
%%%%%%%%%%%%%%%%%%%%%%%%%%%%%%%%%%%%%%%%%%%%%%
%\begin{acks}[Acknowledgments]
%The authors would like to thank the anonymous referees, an Associate
%Editor and the Editor for their constructive comments that improved the
%quality of this paper.
%\end{acks}

%%%%%%%%%%%%%%%%%%%%%%%%%%%%%%%%%%%%%%%%%%%%%%
%% Funding information, if any,             %%
%% should be provided in the                %%
%% funding section.                         %%
%%%%%%%%%%%%%%%%%%%%%%%%%%%%%%%%%%%%%%%%%%%%%%
\section*{Acknowledgments}
This work was supported by the ANII (Uruguay) under Grant FCE-3-2022-1-172289;  CSIC (Uruguay) under Grant 22520220100031UD (first and second authors); MCIN/AEI/10.13039/501100011033 under Grant PID2020-116587GB-I00 and Conseller\'{i}a de Cultura, Educaci\'{o}n e Universidade under Grant ED431C 2021/24 (third author).

%The first and second authors were supported by FCE-3-2022-1-172289, ANII (Uruguay) and 22520220100031UD, CSIC (Uruguay). The third author was supported by grant PID2020-116587GB-I00 funded by MCIN/AEI/10.13039/501100011033 and ED431C 2021/24 funded by Conseller\'{i}a de Cultura, Educaci\'{o}n e Universidade.

\section*{Disclosure statement} 
The authors report there are no competing interests to declare.

\bibliographystyle{abbrv}
\bibliography{biblio.bib}      % Bibliography file (usually '*.bib')

\end{document}